\documentclass[a4paper,leqno,10pt]{amsart}
\usepackage{amsfonts,amssymb,verbatim,amsmath,amsthm,latexsym,textcomp,amscd}
\usepackage{latexsym,amsfonts,amssymb,epsfig,verbatim}
\usepackage{amsmath,amsthm,amssymb,latexsym,graphics,textcomp}
\usepackage{graphicx}
\usepackage{color}
\usepackage{stackrel}
\usepackage{tikz}
\usepackage{url}
\usepackage{algpseudocode}
\usepackage{multirow}
\usepackage{enumerate}
\usepackage[mathscr]{euscript}
\usepackage{txfonts}
\usepackage{euscript}
\usepackage[all]{xy}
\usepackage{epic,eepic}
\input xy
\xyoption{all}

\usepackage{hyperref}

\theoremstyle{plain}
\newtheorem{theorem}{Theorem}[section]
\newtheorem{thm}[theorem]{Theorem}
\newtheorem{lemma}[theorem]{Lemma}

\newtheorem{prop}[theorem]{Proposition}

\newtheorem{remark}[theorem]{Remark}

\theoremstyle{definition}
\newtheorem{defn}[theorem]{Definition}

\newtheorem{exam}[theorem]{Example}

\begin{document}
\title{Cohomology and deformations of oriented dialgebras}
\author{Ali N. A. Koam}
\email{akoum@jazanu.edu.sa}
\address{Department of Mathematics, Faculty of Sciences,
Jazan University, Saudi Arabia.}

\author{Ripan Saha}
\email{ripanjumaths@gmail.com}
\address{Department of Mathematics,
Raiganj University, Raiganj, 733134,
West Bengal, India.}

\begin{abstract}
We introduce a notion of oriented dialgebra and develop a cohomology theory for oriented dialgebras based on the possibility to mix the standard chain complexes computing group cohomology and associative dialgebra cohomology. We also introduce a $1$-parameter formal deformation theory for oriented dialgebras and show that cohomology of oriented dialgebras controls such deformations.
\end{abstract}
\maketitle
\section{Introduction}
The notion of dialgebras was introduced by J.-L. Loday \cite{loday} while studying the periodicity phenomena in algebraic $K$-theory. The associative dialgebras or simply dialgebras are generalizations of associative algebras, whose structure is determined by two associative operations interwined by some relations. A (co)homology theory of dialgebras using the planar binary trees has been introduced by Loday and dialgebra (co)homology with coefficients was introduced by A. Frabetti \cite{Frab}.

In a series of papers, the deformations of associative algebras was developed by Gerstenhaber \cite{G1, G2, G3, G4, G5}. Later following his theory deformations of other algebraic structures have been extensively studied in many research articles, see \cite{NR, NR2, GS} and also references there in. For every one-parameter formal deformations there is a controlling cohomology. For example, the Hochschild cohomology controls deformations of associative algebras, Chevalley-Eilenberg cohomology for Lie algebras, and Harrison cohomology for commutative algebras.

The notion of an oriented algebra was introduced by Koam and Pirashvili  \cite{KP18} with aim to develop the equivariant version of Hochschild
cohomology. Oriented algebras generalize both $G$-algebras and involutive associative algebras. The authors developed a cohomology theory of oriented algebras and showed how extensions and deformations are related to the cohomology.

Algebra with involution plays an important role in the study of algebras arising in the geometric contexts \cite{braun, costello, LV}. In the present paper, we introduce a notion of oriented dialgebra, which is to oriented algebra as a dialgebra is to an associative algebra. Oriented dialgebra is more general than $G$-algebra and involutive dialgebra. Let $\textbf{ODias}$ be the category oriented dialgebras. We discuss about the free object in the category $\textbf{ODias}$. Similar to the non-oriented case, we establish a commutative diagram of functors between categories in the oriented setting.
We mix the standard chain complexes of group cohomology \cite{quillen} and associative dialgebra cohomology \cite{anita gm} to develop a cohomology theory for oriented dialgebra. In the construction of oriented dialgebra cohomology we use mirror images of planar binary trees. We show how cohomology of oriented dialgebras classifies the singular extensions of oriented dialgebras. We also introduce a one-parameter formal deformation for oriented dialgebra which deforms both the multilications and the orientation and show how our cohomology controls the deformation of the oriented dialgebra.

The paper is organized as folllows: In Section \ref{sec 2}, we discuss some preliminaries of dialgebras and its cohomology. We also discuss some results on the plane mirror image of planar binary trees which will be used in the construction of oriented dialgebra cohomology.  In Section \ref{sec 3}, we introduce the notion of oriented dialgebra and give some examples. In Section \ref{sec 4}, we discuss different type of algebras in the oriented setting and how they are related to each other. In Section \ref{sec 5}, we develop the cohomology for the oriented dialgebras. In Section \ref{sec 6}, we study the singular extensions of oriented dialgebras and obtain a standard result about the relationship between extensions and cohomology. In the last Section, we introduce a one-parameter formal deformations of oriented dialgebra and study how the cohomology defined in Section \ref{sec 5} is related to the deformation.

\section{Preliminaries} \label{sec 2}
In this section, we discuss some basics of cohomology of dialgebras which will be used later in this paper. Some references can be found in \cite{anita gm, anita gm2, loday, saha19}. Throughout the paper $\mathbb{K}$ always denotes a field of characteristics zero.

\begin{defn}\label{dialgebra}
A dialgebra over the field $\mathbb{K}$, in short, is a vector space $D$ over $\mathbb{K}$ together with two $\mathbb{K}$-linear associative maps $\dashv,\vdash ~: D\otimes D\to D$, called a left and a right product, satisfying the following axioms:
\begin{align}
\begin{cases}
&(x\dashv y)\dashv z=x\dashv(y\vdash z),\\
&(x\vdash y)\dashv z=x\vdash (y\dashv z),\\
&(x\dashv y)\vdash z=(x\vdash y)\vdash z.
\end{cases}
\end{align}
Associativity of $\dashv$ and $\vdash$ is same as the following axioms:
\begin{align}
\begin{cases}
&(x\dashv y)\dashv z=x\dashv(y\dashv z),\\
&(x\vdash y)\vdash z=x\vdash(y\vdash z).
\end{cases}
\end{align}
\end{defn}
Any associative algebra $(A,\mu)$ has a dialgebra structures by defining a left and a right multiplcation as $x\dashv y=\mu(x,y)= x\vdash y$.

By a morphism of dialgebras $D$ and $D'$, we mean a $\mathbb{K}$-linear map $f:D\to D'$ which preserves both the left and the right multiplications, that is,
\begin{align*}
&f(x\dashv y)=f(x)\dashv f(y),\\
&f(x\vdash y)=f(x)\vdash f(y)\,\,\, \text{for all} \,\,\,x,y\in D.
\end{align*}

\subsection{Cohomology of dialgebras with coefficients}
We discuss cohomology of dialgebras with coefficients in a representation following \cite{anita gm}. Construction of cochain complex uses the concept of a planar binary $n$-tree. So first we recall some basic facts about planar binary $n$-tree. A planar binary $n$-tree is a tree with one root together with $n$ trivalent vertices and $(n+1)$ leaves. The set of all planar binary $n$-trees is denoted by  $Y_n$ and the cardinality of $Y_n$ is the Catalan number $c_n=\frac{(2n)!}{n!(n+1)!}$. For any $y\in Y_n$, we label all the $n+1$ leaves by $\lbrace 0,1,\ldots,n\rbrace$ from the left to the right. The following is the list of all $0,1, 2$ and $3$ trees.
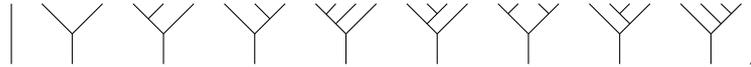
\begin{center}
\begin{tikzpicture}[scale=0.2]
 \draw (-6,0) -- (-4,-2);	\draw (0,0) -- (2,-2);    \draw (6,0) -- (8,-2);  \draw (12,0)-- (14,-2); \draw (14,-2) -- (14,-4); \draw (14, -2) -- (16,0); \draw (12.7, - 0.7) -- (13.3333, 0) ; \draw (13.33, -1.33) -- (14.66, 0);
\draw (-4,-2)-- (-2,0);	\draw (2,-2) -- (4,0);    \draw (8,-2) -- (10,0); \draw (18,0) -- (20,-2); \draw (20, -2) -- (20, -4); \draw (20, -2) -- (22,0); \draw (19.33, -1.33) -- (20.66, 0); \draw (19.33, 0) -- (20, -0.66); \draw (24,0)-- (26,-2); \draw (26,-2)-- (26,-4); \draw (26,-2) -- (28,0); \draw (25.33, 0) -- (24.66, -0.66); \draw (26.66, 0) -- (27.34, -0.66); \draw (30,0) -- (32,-2); \draw (32, -2) -- (32, -4); \draw (32,-2) -- (34, 0); \draw (31.33, 0) -- (32.67 , -1.33) ; \draw (32.66, 0) -- (32, -0.66);
\draw (-4,-2) -- (-4,-4);	\draw (2,-2) -- (2,-4);    \draw (8,-2) -- (8,-4); \draw (36,0) -- (38, -2) ; \draw (38, -2) -- (38, -4) ; \draw (38, -2) -- (40, 0); \draw (37.33, 0) -- ( 38.67, - 1.33) ; \draw (38.66, 0) -- (39.34, -0.66);
\draw (-8,0) -- (-8,-4);	\draw (1,-1) -- (2,0);     \draw (9,-1) -- ( 8,0);
\end{tikzpicture}
.
\end{center}

The grafting of a $p$-tree $y_1$ and a $q$-tree $y_2$ is a $(p+q+1)$-tree and is denoted by $y_1\vee y_2$, which is obtained by joining the roots of $y_1$ and $y_2$ and creating a new root from that vertex.
There is a convention to denote this tree as $[y_1\,\,\, p+q+1\,\,\, y_2]$ in which all the zero elements are ignored except in $Y_0$. With this notation we can write \begin{align*}
&Y_0=\lbrace[0]\rbrace,\\
&Y_1=\lbrace[1]\rbrace,\\
&Y_2=\lbrace [1\,2],[2\,1]\rbrace,\\
&Y_3=\lbrace [1\,2\,3],[2\,1\,3],[1\,3\,1],[3\,1\,2],[3\,2\,1]\rbrace. 
\end{align*}
We will use this notation to write the elements of $Y_n$.

For each $0\leq i\leq n$, there are face and degeneracy maps on the set of $n$-trees $Y_n$ which satisfy almost all the axioms of a simplicial set but $Y_n$ is not a simplicial set with those face and degeneracy maps. For each $0 \leq i \leq n$, face maps are defined by $d_i:Y_n\to Y_{n-1},\,\,\,y\mapsto d_iy$, where $d_iy$ is the $(n-1)$-tree obtained by removing $i^{\text{th}}$ leaf from $y$. For each $0\leq i\leq n$, degeneracy maps are defined by $s_i:Y_n\to Y_{n+1},\,\,\,y\mapsto s_iy$, where $s_iy$ is a $(n+1)$-tree obtained by bifurcating $i^{\text{th}}$ leaf to $y$. The face and degeneracy maps satisfies all the relations of simplicial set except $s_is_i=s_{i+1}s_i$. 

A representation of a dialgebra $D$ is a $\mathbb{K}$-module $M$ endowed with two left actions $\dashv,\vdash: D\otimes M\to M$ and two right actions  $\dashv,\vdash: M\otimes D\to M$  satisfying the axioms of dialgebras, whenever one of the entries $x,y,z\in M$, and the others two are in $D$. One can easily check that $D$ is a representation of itself where left and right actions are induced from the left multiplication and right multiplication of the dialgebra $D$.

Let $\mathbb{K}[Y_n]$ be the $\mathbb{K}$-vector space spanned by the set of $n$-trees $Y_n$. We denote $CY^n(D,M)=\text{Hom}(\mathbb{K}[Y_n]\otimes D^{\otimes n},M                                                                                                                                                                                                                                                                                                                                                                                          )$ for the module of $n$-cochains of $D$ with coefficients in $M$. We define $\delta_i: CY^n(D,M)\to CY^{n+1}(D,M)$ as follows-
$$\label{delta_i} \delta_if(y;a_1,\ldots,a_{n+1})=\begin{cases}
                                                      a_1 \circ^y_0f(d_0y;a_2,\ldots,a_{n+1}),\,\,\,\text{if}\,\,\,i=0,\\
                                                      f(d_iy;a_1,\ldots,a_i \circ^y_ia_{i+1},\ldots,a_{n+1}),\,\,\,\text{if}\,\,\,1\leq i\leq n,\\
                                                      f(d_{n+1}y;a_1,\ldots,a_n) \circ^y_{n+1}a_{n+1},\,\,\,\text{if}\,\,\,i=n+1.
                                                      
                                                      \end{cases} $$
 For each $0\leq i\leq {n+1},\,\,  d_i: Y_{n+1}\to Y_n$ is a face map and for $1\leq i \leq n$, $\circ^y_i:Y_{n+1}\to \lbrace \dashv,\vdash\rbrace$ is defined by
 $$\circ^{[j_1,\ldots,j_{n+1}]}_i=\begin{cases}
                                               \dashv,\,\,\,\text{if}\,\,\, j_i>j_{i+1},\\
                                               \vdash,\,\,\,\text{if}\,\,\,j_i<j_{i+1}.                                  
                                               \end{cases}$$
Thus, for $1\leq i \leq n$, the image of $y\in Y_{n+1}$ under $\circ_i$ is $\circ^y_i=~\dashv$ or $\vdash$ if the $i^{\text{th}}$ leaf points from the vertex is oriented like `$\backslash $' or oriented like `$/$' respectively.

For $i=0, n+1$, $\circ^y_0$ and $\circ^y_{n+1}$ is defined as follows:                                               
$$\circ^{[j_1,\ldots,j_{n+1}]}_0=\begin{cases}
                                               \dashv,\,\,\,\text{if}\,\,\, [j_1,\ldots,j_{n+1}]~\text{is of the form},~[0]\vee y_1~\text{for some}~ n\text{-tree}~ y_1.\\
                                               \vdash,\,\,\,\text{otherwise},                                  
                                               \end{cases},$$ 
                                               and
$$\circ^{[j_1,\ldots,j_{n+1}]}_{n+1}=\begin{cases}
                                               \vdash,\,\,\,\text{if}\,\,\, [j_1,\ldots,j_{n+1}]~\text{is of the form},~y_1\vee [0]~\text{for some}~ n\text{-tree}~ y_1.\\
                                               \dashv,\,\,\,\text{otherwise}.                                  
                                               \end{cases}$$                                                
The coboundary map $\delta^n:  CY^n(D,M)\to CY^{n+1}(D,M)$ is defined as $\delta^n=\sum^{n+1}_{i=0}(-1)^i\delta_i$.
Note that $\delta\circ \delta=0$ due to the  fact that face maps $d_i$ satisfies the relation $d_id_j=d_{j-1}d_i$, see \cite{loday}. The cohomology of the dialgebra $D$ with coefficients in a representation $M$ is defined as
$$HY^n(D,M):=H^n(CY^*(D,M)).$$
In this paper, we consider cohomology of dialgebras $D$ with coefficients over itself. Thus,
$$HY^n(D,D):=H^n(CY^*(D,D)).$$

\subsection{Mirror reflections of planar binary trees} \label{subsec mir}
Let $y$ be a $n$-tree then plane mirror reflection of $y$, denoted by $y^\ast$, is a $n$-tree such that $i^{\text{th}}$ leaf of $y^\ast$ corresponds to $(n-i)^{\text{th}}$ leaf of $y$ and if $i^{\text{th}}$ leaf of $y^\ast$ is oriented like `$\backslash $' then $(n-i)^{\text{th}}$ leaf of $y$ is oriented like `$/$'. For example if we take 
$$y = \begin{tikzpicture}[scale=0.2] \draw (0,0) -- (0, -2); \draw (2,2) -- (0,0) -- (-2,2); \draw (-1,1) -- (0,2); \end{tikzpicture}, ~~~\text{then}~~~y^\ast= \begin{tikzpicture}[scale=0.2] \draw (0,0) -- (0, -2); \draw (2,2) -- (0,0) -- (-2,2);  \draw (1,1) -- (0,2); \end{tikzpicture}.$$
\begin{remark}
It is easy to verify that if  $y= [a_1,\, \ldots, \, a_n]$ then $y^\ast = [a_n,\, \ldots, \, a_1]$. For example if $y = [1\, 2]$ then $y^\ast =[2\, 1]$, if $y = [1\, 2 \, 3]$ then $y^\ast = [3\, 2 \, 1]$. Note that $(y^\ast)^\ast =y$. 
\end{remark}

\begin{lemma}
If $y = y_1 \vee y_2$ then $(y_1 \vee y_2)^\ast = y^\ast_2 \vee y^\ast_1$.
\end{lemma}
\begin{proof}
We can write $y_1 \vee y_2$ as a tree as follows:
\begin{center}
$y_1 \vee y_2 =$
\begin{tikzpicture}[scale=0.2]
 \node[label=]  at (-6, 0) {$y_1$};
 \node[label=]  at (-2, 0) {$y_2$};
 \draw (-6,0) -- (-4,-2);
  \draw (-4,-2)-- (-2,0);	   
\draw (-4,-2) -- (-4,-4);	  
\end{tikzpicture}
\end{center}
On the other hand, taking mirror image of the above tree, we have
\begin{center}
$(y_1 \vee y_2)^\ast$= 
\begin{tikzpicture}[scale=0.2]
 \node[label=]  at (-6, 0) {$y^\ast_2$};
 \node[label=]  at (-2, 0) {$y^\ast_1$};
 \draw (-6,0) -- (-4,-2);
  \draw (-4,-2)-- (-2,0);	   
\draw (-4,-2) -- (-4,-4);	 
\end{tikzpicture}
$= y^\ast_2 \vee y^\ast_1.$
\end{center}
This proves our result.
\end{proof}

\begin{prop} \label{mir prop}
Let $y$ be a planar binary $(n+1)$-tree. For all $0 \leq i \leq n+1$, the following hold:
\begin{enumerate}
\item $(d_iy)\ast = d_{n-i+1}(y^\ast)$.
\item If $\circ^{y}_i =~ \dashv$ then $\circ^{y^\ast}_{n-i+1} = ~\vdash$,  and if $\circ^{y}_i =~ \vdash$ then $\circ^{y^\ast}_{n-i+1} =~ \dashv$.
\end{enumerate}
\end{prop}

\begin{proof}
(1) Let $y$ be a planar $(n+1)$-tree. As the $i^{\text{th}}$ leaf of $y$ corresponds to the $(n-i +1)^{\text{th}}$ leaf of $y^\ast$. Therefore, removing the $i^{\text{th}}$ leaf from $y$ results in removing $(n-i +1)^{\text{th}}$ leaf from $y^\ast$. Thus, our desired result follows.

(2) We only show the first part of the result and proof of the second part is similar to the first part.  For $i = 0$, if $\circ^y_i = ~\dashv$ then $y = | \vee y_1$ for some $n$-tree $y_1$. This implies $y^\ast = y^\ast_1 \vee |$. Therefore, $\circ^{y^\ast}_{n+1} =~ \vdash$. Thus, our claim is true for $i = 0$. Similar calculation shows the result is true for $i = n+1$.

For $1\leq i \leq n$, if $o^y_i = ~\dashv$ then the $i^{\text{th}}$ leaf is oriented like `$\backslash $' . Therefore, the $(n-i +1)^{\text{th}}$ leaf of $y^\ast$ is oriented like `$/$' as $1\leq i \leq n$. Thus, $\circ^{y^\ast}_{n-i+1} =~ \vdash.$
\end{proof}
 
\section{Oriented dialgebras} \label{sec 3}
In this section, we introduce a notion of oriented dialgebras. We give some examples of such algebras. At the end of the section, we discuss how oriented dialgebras are related with other oriented algebras.
\begin{defn}
An involutive dialgebra is a dialgebra $(D, \dashv, \vdash)$ together with a $\mathbb{K}$-linear map $t : D \to D$ satisfying
\begin{enumerate}
\item $t^2 = id_D$;
\item $t(x \dashv y) = t(y) \vdash t(x)$ for all $x, y \in D$.
\end{enumerate}
We may denote an invoultive dialgebra as $(D, t)$.
\end{defn}

\begin{remark}
We call the $\mathbb{K}$-linear map $t$ by an involution on the dialgebra $D$. Observe that
$$t (x \vdash y) = t\big(t^2(x)) \vdash t^2(y) \big) = t^2 (t(y) \dashv y (x)) = t(y) \dashv t(x).$$
Thus, in an involutive dialgebra $(D, t)$ the equality $t (x \vdash y) = t(y) \dashv t(x)$ for all $x, y \in D$ is also hold.
\end{remark}

\begin{defn}
An orientation is a pair $(G,\varepsilon)$, where G is a group and $\varepsilon : G\longrightarrow \lbrace \pm 1 \rbrace$ is a group homomorphism. 
If such orientation is fixed, then we say that $G$ is an oriented group.
\end{defn}
\begin{exam}\hfill 
\begin{enumerate}
\item Any group $G$ can be equipped with a trivial orientation: $\varepsilon(g) = 1$ for all $g\in G$.
\item  For more interesting examples, we could take 
\begin{enumerate}
\item $G = \lbrace \pm 1 \rbrace$ and $\varepsilon = id$.
\item More generally, we can take $G = S_n$ and $\varepsilon(\sigma) = sgn(\sigma)$.
 \end{enumerate}
\end{enumerate}
\end{exam}
\begin{defn}
Let $G$ be a group and $\varepsilon : G \to \lbrace \pm 1 \rbrace$ be a group homomorphism. An oriented dialgebra is a dialgebra $(D, \dashv, \vdash)$ together with an action of $G$ on $D$,
$$G \times D \to D,~~~(g, x)\mapsto gx,$$
such that under this action $D$ is a $G$-module and
\begin{align} \label{orientaion eq1}
g(x \dashv y)=\begin{cases}
            & gx\dashv gy~~~\text{if}~ \varepsilon (g)= +1,\\
            & gy \vdash gx~~~\text{if}~ \varepsilon (g)= -1.
           \end{cases}
\end{align}
\begin{align} \label{orientation eq2}
g(x \vdash y)=\begin{cases}
            & gx\vdash gy~~~\text{if}~ \varepsilon (g)= +1,\\
            & gy \dashv gx~~~\text{if}~ \varepsilon (g)= -1.
           \end{cases}
\end{align}
\end{defn}
If $(D, \dashv, \vdash)$ is an oriented dialgebra equipped with an action of a group $G$ on $D$ with orientation map $\varepsilon$, then we say $D$ is an oriented dialgebra over an oriented group $(G, \varepsilon)$.

\begin{remark}
If $\varepsilon (g) = -1$ then $g (x_1 \circ^{y^\ast}_i x_2) = (gx_1) \circ^y_{n-i+1} (gx_2)$ for $x_1, x_2 \in D, g \in G$. This simply follows from the second part of the Proposition \ref{mir prop}.

If we take $G = \lbrace \pm 1 \rbrace$ and $\varepsilon = id$ then an oriented dialgbra is nothing but an involutive dialgebra. Thus, an oriented dialgebra is more general than an involutive dialgebra.
\end{remark}

\begin{defn}
Let $D_1$ and $D_2$ be oriented dialgebras over an oriented group $(G, \varepsilon)$. A morphism between oriented dialgebras $D_1$ and $D_2$ is a morphism of $G$-modules $f : D_1 \to D_2$ satisfying
\begin{align*}
& f(x \dashv y)= f(x) \dashv f(y),\\
& f(x \vdash y)= f(x) \vdash f(y).
\end{align*}
\end{defn}

The set of oriented dialgebras together with the set of morphisms between them form a category. We denote this category as $\textbf{ODias}$.

\begin{exam}
Any oriented algebra $(A, \mu)$ is an oriented dialgebra by taking $x \dashv y= \mu(x, y)= x \vdash y.$
\end{exam}

\begin{exam}
Suppose $(A, d)$ be an oriented differential associative algebra over an oriented group $(G, \varepsilon)$, that is, $A$ is an oriented associative algebra together with a differential which is compatible with orientation. We define the left and right multiplication on $A$ as follows-
$$x\dashv y:=x\,dy\,\,\,\text{and}\,\,\,x\vdash y:=dx\,y.$$
One can easily verify that with these two operations $(A, d)$ has an oriented dialgebra structure over $(G, \varepsilon)$.
\end{exam}

\begin{exam}
Suppose $A$ be an oriented algebra over an oriented group $(G, \varepsilon)$ and $M$ be an oriented $A$-bimodule. Let $f: M \to A$ be an oriented $A$-bimodule map. Then $M$ can be given an oriented dialgebra structure as follows:
\begin{align*}
& x \dashv y:= xf(y),\\
& x \vdash y:= f(x)y.
\end{align*}
\end{exam}

Let $G$ be an oriented group. Suppose $V$ is a $\mathbb{K}$-linear vector space equipped with an action of $G$.  The tensor module on $V$
$$T(V) = \mathbb{K} \oplus V \oplus V^{\otimes 2} \oplus \cdots \oplus V^{\otimes n} \oplus \cdots$$
with the concatenation product is an oriented associative algebra \cite{DS20}. Consider the following $\mathbb{K}$-module
$$\text{Dias}(V) = T(V) \otimes V \otimes T(V).$$
We write an element $v = v_{-n}\ldots \otimes v_{-1} \otimes v_0\otimes v_1\otimes \ldots v_m$ of $\text{Dias(V)}$ as $v =  v_{-n}\ldots v_{-1} \tilde{v_0} v_1 \ldots v_m$. Let $v, w \in \text{Dias(V)}$ such that
$$v= v_{-n}\ldots v_{-1} ~\tilde{v_0}~ v_1 \ldots v_m,~~~w= w_{-p}\ldots w_{-1}~ \tilde{w_0}~ w_1 \ldots w_q.$$
Then $\text{Dias(V)}$ together with the operations
\begin{align*}
& v \dashv w = v_{-n}\ldots v_{-1}~ \tilde{v_0}~ v_1 \ldots v_m w_{-p}\ldots  \ldots w_q ,\\
& v \vdash w = v_{-n}\ldots  \ldots v_m w_{-p}\ldots w_{-1}~\tilde{w_0}~w_1  \ldots w_q 
\end{align*}
 is a free dialgebra \cite{loday}. Define an action of $G$ on $\text{Dias}(V)$ by
 \begin{align*}
 g (v_{-n}\ldots v_{-1} ~\tilde{v_0}~ v_1 \ldots v_m) = \begin{cases}
                                                                                      & g v_{-n}\ldots g v_{-1} ~g\tilde{v_0}~ gv_1 \ldots gv_m ~~~\text{if}~~~ \varepsilon(g) = +1,\\
                                                        & gv_m \ldots gv_1  ~g\tilde{v_0}~ gv_{-1}\ldots gv_{-n} ~~~\text{if}~~~ \varepsilon(g) = -1.                            
                                                                                         \end{cases}
 \end{align*}
 Then $(\text{Dias(V)}, \dashv, \vdash)$ equipped with this action is a free oriented dialgebra.

\begin{remark}
In \cite{saha19}, the author defined a notion of a finite group action on a dialgebra. Here, we first recall the definition of  group action on a dialgebra $D$.
A group $G$ is said to act on $D$ from the left if there is a map,
\begin{align*}
\phi:~& G\times D\to D,~~ (g, x) \mapsto \phi (g, x) = gx
\setcounter{equation}{0}
\end{align*} 
satisfying the axioms:
\begin{enumerate}
\item $ex=x\,\,\, \text{for all}\,\,\, x\in D\,\,\, \text{and e is the identity of the group G}.$
\item $g_1(g_2x)=(g_1g_2)x\,\,\, \text{for all}\,\,\,x\in D\,\,\, \text{and}\,\,\, g_1,g_2\in G.$
\item $\text{For each}\,\,\, g\in G,\,\,\, \text{the map}\,\,\, \psi_g=\phi(g,-):D\to D\,\,\, \text{is a linear map}.$
\item $gx\dashv gy=g(x\dashv y) \,\,\,\text{and}\,\,\, gx\vdash gy=g(x\vdash y)\,\,\, \text{for all}\,\,\, x,y\in D, g\in G.$ 
\end{enumerate}

Let $G$ be an oriented group with $\varepsilon (g)=1$ for all $g \in G $. Then $G$ acts on $D$ via action defined above. Thus, in this case oriented dialgebras are nothing but equivariant dialgebras. Therefore, the notion of an oriented dialgebra generalizes the notion of a group action on a dialgebra defined in \cite{saha19}. 
\end{remark}

\section{Relation with other type of oriented algebras}\label{sec 4}
In this section we discuss different type of algebras in the oriented setting  and motivated from the Loday's \cite{loday} work on dialgebras we show some results about how they are related to each other.

In \cite{DS20}, the authors introduced the category of oriented dendriform algebras $\textbf{ODend}$ and the category of oriented Lie algebras $\textbf{OLie}$ and also constructed free object in the category of oriented associative algebras $\textbf{OAss}$. Moreover, the authors discussed how $\textbf{OAss}$ is related to $\textbf{OLie}$ and $\textbf{ODend}$.

\begin{prop} \cite{DS20}
The functor $U : OLie \to OAss$ is left adjoint to the functor $(~~)_c : OAss \to OLie$.
\end{prop}

\begin{prop}
Tensor product of an oriented dialgebra and an oriented dendriform algebra with the bracket
$$[x\otimes a, y \otimes b] = (x \dashv y) \otimes (a\prec b) - (y\vdash x) \otimes (b \succ a) - (y \dashv x) \otimes (b \prec a) + (x \vdash y) \otimes (a \succ b)$$
defines a structure of an oriented Lie algebra.
\end{prop}
\begin{proof}
Tensor product of a dialgebra and a dendriform algebra with the given bracket has a structure of a Lie algebra \cite{loday}. We need to only check that the given bracket preserve the action of the orientation group $(G, \varepsilon)$ on the tensor product. We assume $G$ acts on the tensor product component wise.
For $\varepsilon (g) = +1$, observe that
\begin{align*}
& g [x \otimes a, y\otimes b] \\
& = g \big((x \dashv y) \otimes (a\prec b) - (y\vdash x) \otimes (b \succ a) - (y \dashv x) \otimes (b \prec a) + (x \vdash y) \otimes (a \succ b) \big) \\
& = (gx \dashv gy) \otimes (ga\prec gb) - (gy\vdash gx) \otimes (gb \succ ga) - (gy \dashv gx) \otimes (gb \prec ga) + (gx \vdash gy) \otimes (ga \succ gb) \\
&= [gx \otimes ga, gy \otimes gb].
\end{align*}
Next for $\varepsilon (g) = -1$, we have
\begin{align*}
& g [x \otimes a, y\otimes b] \\
& = g \big((x \dashv y) \otimes (a\prec b) - (y\vdash x) \otimes (b \succ a) - (y \dashv x) \otimes (b \prec a) + (x \vdash y) \otimes (a \succ b) \big) \\
& = (gy \vdash gx)\otimes (gb \succ ga) - (gx \dashv gy) \otimes (ga \prec gb) - (gx \vdash gy) \otimes (ga \succ gb) + (gy \dashv gx) \otimes (gb \prec ga) \\
& = [gy\otimes gb, gx \otimes ga].
\end{align*}
Thus, the tensor product of an oriented dialgebra and an oriented dedriform algebra is an oriented Lie algebra.
\end{proof}

Recall, a Leibniz algebra is a vector space $L$ over $\mathbb{K},$ equipped with a bracket operation, which is $\mathbb{K}$-bilinear and satisfies the Leibniz identity: 
$$[x,[y,z]]= [[x,y],z]-[[x,z],y] ~~\mbox{for}~x,~y,~z \in L.$$

A dual Leibniz algebra or a Zinbiel algebra is a $\mathbb{K}$-vector space $R$ equipped with a bilinear map
$$.~: R \times R \rightarrow R$$ satisfying the relation
\begin{align}\label{zinbiel-relation}
((r.s).t) = (r.(s.t)) + (r.(t.s)),  \forall~r, s, t \in R.
\end{align} 

\begin{defn}
An oriented Leibniz algebra over an oriented group $(G, \varepsilon)$ consists of a Leibniz algebra $(L, [~~,~~])$ together with an action of $G$,
$$G\times L \to L,~~~(g, x)\mapsto gx,$$
satisfying
\begin{align*}
g [x ,y] = \begin{cases}
                 &~[gx, gy],~~~\text{if}~~~\varepsilon(g) = +1,\\
                 &- [gx, gy],~~~\text{if}~~~\varepsilon(g) = -1.
                \end{cases}
\end{align*}
\end{defn}
Note that any oriented Lie algebra is automatically an oriented Leibniz algebra.

One can define morphism between oriented Leibniz algebras in a similar way as for oriented dialgebras. Let $\textbf{OLeib}$ denotes the category whose objects are oriented Leibniz algebras and morphisms are oriented Leibniz algebra morphisms.

The following proposition is an oriented version of the well known result connecting dialgebras and Leibniz algebras.
\begin{prop} \label{di Leib}
Let $D$ be an oriented dialgebra.  Then $D$ together with the bracket  $[x, y]_D = x \dashv y - y \vdash x$ is an oriented Leibniz algebra.
\end{prop}
\begin{proof}
For any dialgebra $D$ the bracket $[x, y]_D = x \dashv y - y \vdash x$ makes $D$ a Leibniz algebra \cite{loday}. We only need to show that this bracket preserves the action of the oriented group $G$.
\begin{align*}
g [x ,y]_D = \begin{cases}
                 &g ( x \dashv y - y \vdash x) = gx \dashv gy - gy \vdash gx =[gx, gy]_D,~~~\text{if}~~~\varepsilon(g) = +1,\\
                 &g ( x \dashv y - y \vdash x) = gy \vdash gx - gx \dashv gy = -[gx, gy]_D,~~~\text{if}~~~\varepsilon(g) = -1.
                \end{cases}
\end{align*}
\end{proof}
We denote the oriented Leibniz algebra $(D, [~~,~~]_D)$ as $D_{OLeib}$. From the Proposition (\ref{di Leib}) we have a functor
$$[~~,~~]_D : \textbf{ODias} \longrightarrow \textbf{OLeib}.$$

Let $L$ be an oriented Leibniz algebra. From the previous discussion
$$\text{Dias} (L) = T(L)\otimes L \otimes T(L)$$
has an oriented dialgebra structure. Consider
$$Ud(L) = \text{Dias} (L)/\lbrace [x, y] - x\dashv y + y \vdash x \mid x, y \in L \rbrace.$$
The orientation of $L$ and $\text{Dias} (L)$ induces an orientation on $Ud(L)$ as 
\begin{align*}
g ([x, y] - x\dashv y + y \vdash x)= \begin{cases}
                                                            & ~[gx, gy] - gx\dashv gy + gy \vdash gx,~~~\text{if}~~~\varepsilon(g) = +1 \\
                                                            & -[gx, gy] - gy\vdash gx + gx \dashv gy,~~~\text{if}~~~\varepsilon(g) = -1 .
                                                           \end{cases}
\end{align*}
Therefore, $Ud(L)$ is an oriented dialgebra.
This construction gives us a functor:
$$\text{Ud} :  \textbf{OLeib} \to \textbf{ODias}.$$

\begin{prop}
The functor $\text{Ud} :  \textbf{OLeib} \to \textbf{ODias}$ is left adjoint to the functor $[~~,~~]_D : \textbf{ODias} \longrightarrow \textbf{OLeib}.$ Thus, we have
$$\text{Hom}_\textbf{ODias} (Ud(L), D) \cong \text{Hom}_{OLeib}(L, D_{OLeib}).$$  
\end{prop} 

\begin{proof}
Let $f : L \to D_{OLeib}$ be a morphism of oriented Leibniz algebras. Then $f$ extends uniquely to a morphism of oriented dialgebras from $\text{Dias}(L)$ to $D$. As the image of $[x, y] - x\dashv y + y \vdash x$ is zero under this morphism, we have a morphism from $Ud(L)$ to $D$.

Conversely, for any oriented morphism $g : Ud(L) \to D$, the restriction of $g$ to $L$ gives us a morphism from $L$ to $D_{OLeib}$. 

\end{proof}

\begin{defn}
An oriented Zinbiel algebra over an oriented group $(G, \varepsilon)$ consists of a Zinbiel algebra $(R, .)$ together with an action of $G$,
$$G\times R \to R,~~~(g, x)\mapsto gx,$$
satisfying
$$ g (x.y) =gx. gy~~~\text{for all}~~~\varepsilon(g).$$
We denote the category of all oriented Zinbiel algebras by $\textbf{OZinb}$.
\end{defn}

\begin{lemma}
Let $(R,.)$ be an oriented Zinbiel algebra over the oriented group $(G, \varepsilon)$. Then $R$ together with the multiplication
$$x \prec y := x.y,~~~x \succ y := y.x,~~~\text{for all}~~~x, y \in R,$$
defines an oriented dendriform algebra.
\end{lemma}
\begin{proof}
It is an established result that $(R, \prec, \succ)$ is a dendriform algebra and it is easy to check that the defined multiplications preserve the action. Thus, $(R, \prec, \succ)$ is an oriented dendriform algebra
\end{proof}

\begin{defn}
An oriented commutative algebra over an oriented group $(G, \varepsilon)$ consists of a commutative algebra $\big(A, (~~)\big)$ together with an action of $G$,
$$G\times A \to A,~~~(g, x)\mapsto gx,$$
satisfying
$$ g (xy) =(gx gy)~~~\text{for all}~~~\varepsilon(g).$$
\end{defn}
The set of all oriented commutative algebra together with the natural morphism between them forms a category, we call this category as $\textbf{OComm}$.

\begin{lemma}
Let $(R,.)$ be an oriented Zinbiel algebra over the oriented group $(G, \varepsilon)$. Then the symmetric product
$$(xy) = x.y + y.x,$$
is an oriented commutative and associative algebra.
\end{lemma}

From the above discussion, it is clear that similar to the classical case we have the following diagram in the oriented case.

$$
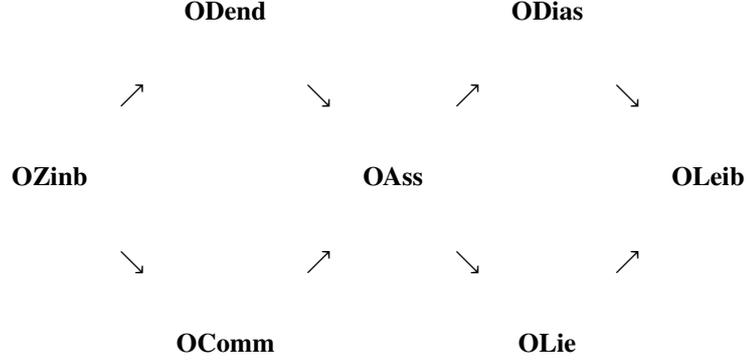
\xymatrixrowsep{.08in}
\xymatrixcolsep{.08in}
\xymatrix{
&&{\bf ODend}&&&&{\bf ODias}&\cr
&&&&&&&\cr
&\nearrow&&\searrow&&\nearrow&&\searrow\cr
&&&&&&&\cr
{\bf OZinb}&&&&{\bf OAss}&&&&{\bf OLeib}\cr
&&&&&&&\cr
&\searrow&&\nearrow&&\searrow&&\nearrow\cr
&&&&&&&\cr
&&{\bf OComm}&&&&{\bf OLie}&\cr
}
$$

\section{Cohomology of oriented dialgebras}\label{sec 5}
In this section, we develop a cohomology theory for oriented dialgebras using chain complexes computing associative dialgebra cohomologies \cite{anita gm, anita gm2} and group cohomology. we use plane mirror image of binary tress to construct the cochain complex for oriented dialgebra cohomology.
\begin{defn}
Let $D$ be an oriented dialgebra over an oriented group $(G,\varepsilon)$. An $oriented$ $bimodule$ over $D$ is an usual bimodule $M$ together with a $G$-module structure on $M$ such that
\begin{align*}
& g (x \dashv m) =\begin{cases}
                           &gx \dashv gm, ~~~~\text{if}~~~\varepsilon(g) = +1,\\
                          & gm \vdash gx,~~~~\text{if}~~~\varepsilon(g) = -1,
                           \end{cases}                           
                          \hspace{0.5in} g (x \vdash m) =\begin{cases}
                           &gx \vdash gm, ~~~~\text{if}~~~\varepsilon(g) = +1, \\
                          & gm \dashv gx,~~~~\text{if}~~~\varepsilon(g) = -1,
                           \end{cases} \\
&g (m \dashv x) =\begin{cases}
                           &gm \dashv gx, ~~~~\text{if}~~~\varepsilon(g) = +1,\\
                          & gx \vdash gm,~~~~\text{if}~~~\varepsilon(g) = -1,
                           \end{cases}                           
                          \hspace{0.5in} g (m \vdash x) =\begin{cases}
                           &gm \vdash gx, ~~~~\text{if}~~~\varepsilon(g) = +1, \\
                          & gx \dashv gm,~~~~\text{if}~~~\varepsilon(g) = -1.
                           \end{cases}                           
\end{align*}

\end{defn}
Let D be an oriented dialgebra over an oriented group $(G,\varepsilon)$ and let $M$ be an oriented bimodule. For any $n\geqslant 0$, we define an action of $G$ on $CY^n (D, M)= \text{Hom}(\mathbb{K}[Y_n]\otimes D^{\otimes n},M                                                                                                                                                                                                                                                                                                                                                                                          )$ by
\begin{equation}\label{action on cochain}
 (gf)(y;x_1,\ldots,x_n) = \begin{cases}
gf(y;g^{-1} x_1,\ldots,g^{-1} x_n), &  {\rm if} \  \varepsilon(g) = +1,\\
 (-1)^{\frac{(n-1)(n-2)}{2}}\; gf(y^\ast; g^{-1} x_n,\ldots, g^{-1} x_1), & {\rm if} \ \varepsilon(g) = -1, \end{cases}\end{equation}
where $x_1,\ldots, x_n \in D,~y \in Y_n$ and $y^\ast$ denotes the mirror refection of the $n$-tree $y$. In particular, for $n=1$ the action is independent on the parity of $\varepsilon(g)$. We denote $CY^n (D, M)$ equipped with an action of $G$ by $C^n (D, M).$
\begin{remark}
It is easy to see that any oriented dialgebra $D$ is an oriented bimodule over itself. In this paper, we shall consider $D$ as a bimodule over itself.
\end{remark}
\begin{prop}
The coboundary map $\delta : CY^n(D, D) \to CY^{n+1}(D, D)$ defined in \ref{delta_i} satisfies $\delta (g f)= g \delta (f)$, that is , $\delta$ is an equivariant map. Thus, $\delta$ induces a map  $\delta : C^n(D, D) \to C^{n+1}(D, D)$.
\end{prop}
\begin{proof}
We first prove the proposition for the case of $\varepsilon(g) = +1$. For $x_1,x_2,\ldots, x_{n+1} \in D$ and $g\in G$, we have
\begin{align*}
& \delta (gf)(y;x_1, \ldots, x_{n+1})\\
&= x_1 \circ^y_0 (gf) (d_0y;x_2,\ldots,x_{n+1})\\
&+ \sum^n_{i=1}(-1)^i(gf)(d_i y; x_1,\ldots,x_i \circ^y_i x_{i+1}, \ldots, x_{n+1})\\
&+(-1)^{n+1} (gf)(d_{n+1}y;x_1,\ldots,x_n)\circ^y_{n+1}x_{n+1}. \\           
&= x_1 \circ^y_0 g (f(d_0 y;g^{-1} x_2,\ldots, g^{-1} x_{n+1}))\\
&+ \sum^n_{i=1} (-1)^i g (f(d_i y; g^{-1}x_1,\ldots, g^{-1 }x_i \circ^y_i g^{-1} x_{i+1},\ldots, g^{-1} x_{n+1}))\\
& +(-1)^{n+1}g (f(d_{n+1}y;g^{-1} x_1,\ldots, g^{-1} x_n)) \circ^y_{n+1} x_{n+1}.
\end{align*}
We also have
\begin{align*}
& (g\delta (f))(y;x_1, \ldots, x_{n+1})\\
&=g (\delta (f))(y; g^{-1 }x_1, \ldots,g^{-1} x_{n+1})\\
&=g (g^{-1 }x_1 \circ^y_0 f (d_0 y; g^{-1} x_2,\ldots,g^{-1} x_{n+1}))\\
& + \sum^n_{i=1} (-1)^i g (f(d_i y; g^{-1} x_1,\ldots,g^{-1} x_i \circ^y_i g^{-1}x_{i+1},\ldots,g^{-1} x_{n+1}))\\
&+(-1)^{n+1}g(f (d_{n+1}y; g^{-1} x_1,\ldots,g^{-1 }x_n) \circ^y_{n+1} g^{-1} x_{n+1}).                                                     
\\
&= x_1 \circ^y_0 g f (d_0 y; g^{-1} x_2,\ldots,g^{-1} x_{n+1}))\\
& +  \sum^n_{i=1} (-1)^i g (f(d_i y; g^{-1} x_1,\ldots,g^{-1} x_i \circ^y_i g^{-1}x_{i+1},\ldots,g^{-1} x_{n+1}))\\
&+(-1)^{n+1} g f (d_{n+1}y; g^{-1} x_1,\ldots,g^{-1 }x_n)\circ^y_{n+1} x_{n+1}.                                                     
\end{align*}
Thus, we have $\delta (g f)= g \delta (f)$ for $\varepsilon (g) = +1$. Next we need to deal with $\varepsilon(g) = -1$. Observe that
\begin{align*}
& \delta (gf)(y;x_1, \ldots, x_{n+1})\\
&= x_1 \circ^y_0 (gf) (d_0y;x_2,\ldots,x_{n+1})\\
&+ \sum^n_{i =1} (-1)^i (gf)(d_i y; x_1,\ldots,x_i \circ^y_i x_{i+1}, \ldots, x_{n+1})\\
&+(-1)^{n+1} (gf)(d_{n+1}y;x_1,\ldots,x_n) \circ^y_{n+1}x_{n+1}. \\           
&= (-1)^{\frac{(n-1)(n-2)}{2}}x_1 \circ^y_0 g (f((d_0 y)^\ast;g^{-1} x_{n+1},\ldots, g^{-1} x_{2}))\\
&(-1)^{\frac{(n-1)(n-2)}{2}}  \sum^n_{i=1}(-1)^{i} g (f( (d_i y)^\ast; g^{-1}x_{n+1},\ldots, g^{-1 }x_{i+1} \circ^{y^\ast}_{n-i+1} g^{-1} x_i,\ldots, g^{-1} x_1))\\
& (-1)^{\frac{(n-1)(n-2)}{2}+ (n+1)}g (f( (d_{n+1}y)^\ast;g^{-1} x_n,\ldots, g^{-1} x_1)) \circ^y_{n+1} x_{n+1}.
\end{align*}
On the other hand, we have
\begin{align*}
& (g\delta (f))(y;x_1, \ldots, x_{n+1})\\
&=(-1)^{\frac{n(n-1)}{2}}g (\delta (f))(y^\ast; g^{-1 }x_{n+1}, \ldots,g^{-1} x_{1})\\
&=(-1)^{\frac{n(n-1)}{2}} g (g^{-1 }x_{n+1} \circ^{y^\ast} _0 f (d_0 y^\ast; g^{-1} x_n,\ldots,g^{-1} x_{1}))\\
& + (-1)^{\frac{n(n-1)}{2}}\sum^n_{i=1}(-1)^{i}  g (f(d_{i} y^\ast; g^{-1} x_{n+1},\ldots,g^{-1} x_{n-i+2} \circ^{y^\ast}_{i} g^{-1}x_{n-i+1},\ldots,g^{-1} x_{1}))\\
&+(-1)^{\frac{n(n-1)}{2}+(n+1)}g(f (d_{n+1}y^\ast; g^{-1} x_{n+1},\ldots,g^{-1 }x_2)\circ^{y^\ast}_{n+1} g^{-1} x_{1}).                                                     
\\
&=(-1)^{\frac{n(n-1)}{2}}  g (f (d_0 y^\ast; g^{-1} x_n,\ldots,g^{-1} x_{1})) \circ^{y}_{n+1} x_{n+1}\\
& + (-1)^{\frac{n(n-1)}{2}}\sum^n_{j=1}(-1)^{n-j+1}  g (f(d_{n-j+1} y^\ast; g^{-1} x_{n+1},\ldots,g^{-1} x_{j+1} \circ^{y^\ast}_{n-j+1} g^{-1}x_{j},\ldots,g^{-1} x_{1}))\\
&+(-1)^{\frac{n(n-1)}{2}+(n+1)} x_{1} \circ^{y}_{0} g(f (d_{n+1}y^\ast; g^{-1} x_{n+1},\ldots,g^{-1 }x_2)).
\end{align*}
This proves our proposition.
\end{proof}

 Thus, one can form the following bicomplex ${C}{^\ast_G}(D, D)$.
\begin{displaymath}
\xymatrix{
   \vdots               & \vdots                    & \vdots                    & \\
\mathcal{M}aps(G^2, D) \ar[u]^{\partial^{''}} \ar[r]^-{\partial^{'}} & \mathcal{M}aps(G^2,\text{Hom}(\mathbb{K}[Y_1]\otimes D, D)) \ar[r]^-{\partial^{'}} \ar[u]^{\partial^{''}}  & \mathcal{M}aps(G^2,\text{Hom}(\mathbb{K}[Y_2]\otimes D^{\otimes 2},D)) \ar[u]^-{\partial^{''}} \ar[r]^-{\partial^{'}} & \cdots \\ 
 \mathcal{M}aps(G,D) \ar[u]^-{\partial^{''}} \ar[r]^-{\partial^{'}} & \mathcal{M}aps(G,\text{Hom}(\mathbb{K}[Y_1]\otimes D,D)) \ar[r]^-{\partial^{'}} \ar[u]^-{\partial^{''}} & \mathcal{M}aps(G,\text{Hom}(\mathbb{K}[Y_2]\otimes D^{\otimes 2},D)) \ar[u]^-{\partial^{''}} \ar[r]^-{\partial^{'}} & \cdots \\
 D \ar[u]^-{\partial^{''}} \ar[r]^-{\partial^{'}} & \text{Hom}(\mathbb{K}[Y_1]\otimes D, D) \ar[r]^-{\partial^{'}} \ar[u]^-{\partial^{''}} &\text{Hom}(\mathbb{K}[Y_2]\otimes D^{\otimes 2},D) \ar[u]^-{\partial^{''}} \ar[r]^-{\partial^{'}} & \cdots  \\
}\end{displaymath}
The coboundary maps are given as follows:
\begin{itemize}
\item The coboundary of every horizontal maps $\partial^{'}$ are given by:
\begin{equation*}
\begin{split}
(\partial^{'}\alpha)(g_1,\ldots,g_n; y; x_1,\ldots,x_{n+1})& = x_1 \circ^y_0 \alpha(g_1,\ldots,g_n; d_0 y; x_2,\ldots,x_{n+1})\\& + \sum^n_{ i=1} (-1)^i \alpha(g_1,\ldots,g_n; d_i y; x_1,\cdots,x_i \circ^y_i x_{i+1},\ldots, x_{n+1})\\& + (-1)^{n+1} \alpha(g_1,\ldots,g_n; d_{n+1}y; x_1,\ldots,x_n) \circ^y_{n+1}x_{n+1}.
\end{split}
\end{equation*}
\item The coboundary of the first vertical maps are given by:
\begin{equation*}
\begin{split}
(\partial^{''}f)(g_1,\ldots,g_{n+1})& = g_1 f(g_2,\ldots,g_{n+1})\\& + \sum_{i=1}^n (-1)^i f(g_1,\ldots,g_ig_{i+1},\ldots,g_{n+1})\\& + (-1)^{n+1} f(g_1,\ldots,g_n).
\end{split}
\end{equation*}
\item The coboundary of the second vertical maps when $\varepsilon(g) = \pm 1$ are given by:
\begin{equation*}
\begin{split}
(\partial^{''}\beta)(g_1,\ldots,g_{n+1};y;x)& = {{{g_1}}\beta}(g_2,\ldots,g_{n+1};y;{g_{1}^{-1}{x}})\\& + \sum_{i=1}^n (-1)^i \beta(g_1,\ldots,g_ig_{i+1},\ldots,g_{n+1};y;x)\\& + (-1)^{n+1} \beta(g_1,\ldots,g_n;y;x).
\end{split}
\end{equation*}
\item The coboundary of the third vertical maps when $\varepsilon(g) = +1$ are given by:
\begin{equation*}
\begin{split}
(\partial^{''}\gamma)(g_1,\ldots,g_{n+1};y;x_1,x_2)& = {{{g_1}}\gamma}(g_2,\ldots,g_{n+1};y;{g_{1}^{-1}{x_1}},{g_{1}^{-1}{x_2}})\\& + \sum_{i=1}^n (-1)^i \gamma(g_1,\ldots,g_ig_{i+1},\ldots,g_{n+1};y;x_1,x_2)\\& + (-1)^{n+1} \gamma(g_1,\ldots,g_n;y;x_1,x_2).
\end{split}
\end{equation*}
\item The coboundary of the third vertical maps when $\varepsilon(g) = -1$ is given by:
\begin{equation*}
\begin{split}
(\partial^{''}\gamma)(g_1,\ldots,g_{n+1};y;x_1,x_2)& = {{{g_1}}\gamma}(g_2,\ldots,g_{n+1};y^\ast;{g_{1}^{-1}{x_2}},{g_{1}^{-1}{x_1}})\\& + \sum_{i=1}^n (-1)^i \gamma(g_1,\ldots,g_ig_{i+1},\ldots,g_{n+1};y;x_1,x_2)\\& + (-1)^{n+1} \gamma(g_1,\ldots,g_n;y;x_1,x_2).
\end{split}
\end{equation*}
\end{itemize}
We will also need the following double complex $\tilde{C}{^\ast_G}(D,D)$ which is obtained by deleting the first column and reindexing. 

\begin{displaymath}
\xymatrix{
   \vdots               & \vdots                    & \vdots                    & \\
 \mathcal{M}aps(G^2, \text{Hom}(\mathbb{K}[Y_1]\otimes D, D)) \ar[r]^-{\partial^{'}} \ar[u]^{\partial^{''}}  & \mathcal{M}aps(G^2, \text{Hom}(\mathbb{K}[Y_2]\otimes D^{\otimes 2},D)) \ar[u]^-{\partial^{''}} \ar[r]^-{\partial^{'}} & \cdots \\ 
  \mathcal{M}aps(G, \text{Hom}(\mathbb{K}[Y_1]\otimes D,D)) \ar[r]^-{\partial^{'}} \ar[u]^-{\partial^{''}} & \mathcal{M}aps(G, \text{Hom}(\mathbb{K}[Y_2]\otimes D^{\otimes 2},D)) \ar[u]^-{\partial^{''}} \ar[r]^-{\partial^{'}} & \cdots \\
 \text{Hom}(\mathbb{K}[Y_1]\otimes D, D) \ar[r]^-{\partial^{'}} \ar[u]^-{\partial^{''}} & \text{Hom}(\mathbb{K}[Y_2]\otimes D^{\otimes 2},D) \ar[u]^-{\partial^{''}} \ar[r]^-{\partial^{'}} & \cdots  \\
}\end{displaymath}
\begin{defn}
We define
$$\tilde{C}^n_G(D,D) = \bigoplus_{\substack{i+j= n+1\\ i \geq 0, j \geq 1}} \mathcal{M}aps(G^i, \text{Hom}(\mathbb{K}[Y_j]\otimes D^{\otimes j},D))~~~\text{for}~~~n\geq 0.$$
The coboundary map $\partial : \tilde{C}^n_G(D,D) \to \tilde{C}^{n+1}_G(D,D)$ is defined as
$$\partial (f) = \partial^{''}(f) + (-1)^i \partial^{'}(f)~~~\text{for}~~~ f \in \mathcal{M}aps(G^i, \text{Hom}(\mathbb{K}[Y_j]\otimes D^{\otimes j},D)).$$
The homologies of the total complex of $\tilde{C}{^\ast_G}(D,D)$ are denoted by $\tilde{H}{^n_G}(D,D)$ where $n\geqslant 0$, and called the cohomology of the oriented dialgebra.
\end{defn}
Observe that $\tilde{H}^1_G(D,D)= \tilde{Z}^1_G(D,D)/\tilde{B}^1_G(D,D)$, where $\tilde{Z}^1_G(D,D)$ is the collection of pairs $(\alpha,\beta)$. Here $\alpha\in  \mathcal{M}aps(G, \text{Hom}(\mathbb{K}[Y_1]\otimes D,D))$ and $\beta \in \text{Hom}(\mathbb{K}[Y_2]\otimes D^{\otimes 2},D)$. Since $\text{Hom}_\mathbb{K}(\mathbb{K}[Y_1]\otimes_\mathbb{K} D, D) \simeq \text{Hom}_\mathbb{K}(D,D)$, we may write $\alpha(g, x)$ instead of $\alpha(g; [1]; x)$. The pair $(\alpha, \beta)$ satisfy the following conditions:
\begin{align}\label{coho 1}
\alpha(gh, x) = {g\alpha(h, {{g^{-1}}{x}})} + \alpha(g, x),
\end{align}
\begin{align}\label{2-cocycle exp}
&x_1 \circ^y_0\alpha(g;d_0 y;x_2) - \alpha(g;d_1y;x_1 \circ^y_1 x_2) + \alpha(g;d_2y;x_1) \circ^y_2x_2 \\
&= \begin{cases} \nonumber
g\beta(y;{g^{-1}}x_1,{{g^{-1}}{}x_2)}-  \beta(y;x_1,x_2), &  {\rm if}\ \varepsilon(g) = +1,\\
g\beta(y^\ast;{g^{-1}}{}x_2,{g^{-1}}{}x_1) - \beta(y;x_1,x_2), & {\rm if} \  \varepsilon(g) = -1, \end{cases}
\end{align}
and
\[ \partial{'}(\beta)= \delta(\beta) = x_1 \circ^y_0\beta(d_0y;x_2,x_3) - \beta(d_1y;x_1\circ^y_1x_2,x_3) + \beta(d_2y;x_1,x_2 \circ^y_2x_3) - \beta(d_3y;x_1,x_2) \circ^y_3x_3 = 0.\]
Note that the last condition implies $\beta$ is a $2$-cocycle of the dialgebra cohomology defined in \cite{anita gm}.
Moreover, $(\alpha,\beta)\in \tilde{B}^1_G(D,D)$ if and only if there exists $\gamma \in \text{Hom}(\mathbb{K}[Y_1]\otimes D,D)$ such that 
\[\beta(y;x_1,x_2) = x_1 \circ^y_0\gamma(d_0y;x_2) - \gamma(d_1y;x_1 \circ^y_1x_2) + \gamma(d_2y;x_1) \circ^y_2x_2\]
and
\[\alpha(g,x) = {g\gamma({g^{-1}}{x})} - \gamma(x).\]

We denote $\beta([2~1]; x_1, x_2)=\beta^l (x_1, x_2)$ and $\beta([1~2]; x_1, x_2)=\beta^r (x_1, x_2)$. Now using this notation we can rewrite the Equation \ref{2-cocycle exp} as follows:
\begin{align}\label{2-cocycle exp 1}
&x_1\dashv \alpha(g, x_2) - \alpha(g, x_1 \dashv x_2) + \alpha(g, x_1)\dashv x_2 \\
&= \begin{cases} \nonumber
g \beta^l ({g^{-1}}x_1,{{g^{-1}}{}x_2)} - \beta^l (x_1,x_2), &  {\rm if}\ \varepsilon(g) = +1,\\
g\beta^r ({g^{-1}}{}x_2,{g^{-1}}{}x_1) - \beta^l (x_1,x_2), & {\rm if} \  \varepsilon(g) = -1, \end{cases}
\end{align}
and
\begin{align}\label{2-cocycle exp 2}
&x_1\vdash \alpha(g, x_2) - \alpha(g, x_1 \vdash x_2) + \alpha(g, x_1)\vdash x_2 \\
&= \begin{cases} \nonumber
g \beta^r ({g^{-1}}x_1,{{g^{-1}}{}x_2)} - \beta^r (x_1,x_2), &  {\rm if}\ \varepsilon(g) = +1,\\
g\beta^l ({g^{-1}}{}x_2,{g^{-1}}{}x_1) - \beta^r (x_1,x_2), & {\rm if} \  \varepsilon(g) = -1. \end{cases}
\end{align}
 For $\varepsilon (g)= +1$, we substitute $g^{-1}x_1=y_1,~g^{-1}x_2=y_2$ and for $\varepsilon (g)= -1$, we substitute $g^{-1}x_1=y_2,~g^{-1}x_2=y_1$ in Equations \ref{2-cocycle exp 1} and \ref{2-cocycle exp 2},  and rearranging first equation of (\ref{2-cocycle exp 1}) and second equation of (\ref{2-cocycle exp 2}) and similarly rest two equations, we get
\begin{align}\label{2-cocycle exp 3}
g\beta_1^l(y_1, y_2) 
=\begin{cases}
& \beta^l_1(g y_1, g y_2) + (\alpha_1(g, g y_1) \dashv g y_2) + (g y_1 \dashv \alpha_1 (g, g y_2)) - \alpha_1(g, g y_1 \dashv g y_2)~~~\text{if}~\varepsilon (g)=+1, \\
& \beta^r_1 (g y_2, g y_1) + (\alpha_1 (g, g y_2)\vdash g y_1) + (g y_2\vdash \alpha_1 (g, g y_1))- \alpha_1(g, g y_2 \vdash g y_1~~~\text{if}~\varepsilon (g)=-1,
\end{cases}    
\end{align}
and
\begin{align}\label{2-cocycle exp 4}
g\beta_1^r(y_1, y_2)  &=\begin{cases}
                                              & \beta^r_1(g y_1, g y_2) + (\alpha_1(g, g y_1) \vdash g y_2) + (g y_1 \dashv \alpha_1 (g, g y_2))- \alpha_1(g, g y_1 \vdash g y_2)~~~\text{if}~\varepsilon (g)=+1,\\
                                              & \beta^l_1 (g y_2, g y_1) + (\alpha_1 (g, g y_2)\dashv g y_1) + (g y_2\dashv \alpha_1 (g, g y_1)) - \alpha_1(g, g y_2 \dashv g y_1)~~~\text{if}~\varepsilon (g)=-1.
                                             \end{cases} 
\end{align}

\section{Classification of singular extensions of oriented dialgebras} \label{sec 6}
It is a well-known fact that the second Hochschild cohomology classifies the singular extensions of associative algebras. Here we obtain  a similar result for oriented dialgebras.
\begin{defn}
Let $D$ be an oriented dialgebra over an oriented group $(G,\varepsilon)$. Moreover, consider $D$ as an oriented bimodule over itself. A singular extension of $D$ by $D$ is a $\mathbb{K}$-split short exact sequence of $G$-modules 
 \[0 \to D\xrightarrow{i} B \xrightarrow{p} D \xrightarrow{}  0, \]
where $B$ is also an oriented dialgebra over  $(G,\varepsilon)$. Furthermore, $p$ and $i$ are morphisms of oriented dialgebras such that
 \[i(x_1) \dashv i(x_2)= i(x_1) \vdash i(x_2)= 0,\]
 \[i(x) \dashv b=i(x \dashv p(b)),\]
  \[i(x) \vdash b=i(x \vdash p(b)),\]
 \[b \dashv i(x)=i(p(b) \dashv x),\]
  \[b \vdash i(x)=i(p(b) \vdash x),\]
 for all $x,x_1,x_2\in D$, $b\in B$.
\end{defn}
\begin{thm}\label{gaf}
Let $D$ be an oriented dialgebra over an oriented group $(G,\varepsilon)$. Then there is a one-to-one correspondence between equivalence classes of extensions of $D$ by $D$ and $\tilde{H}^1_G(D,D)$.
 \end{thm}

 \begin{proof}
Let us start with a singular extension as above. We can think $D$ as a submodule of $B$ via the inclusion map $i(x)=x$. Choose a linear map $s:D\longrightarrow B$ such that $ps=id_D$. We define 
$$\alpha\in  \mathcal{M}aps(G, \text{Hom}(\mathbb{K}[Y_1]\otimes D,D)),~~~\beta \in \text{Hom}(\mathbb{K}[Y_2]\otimes D^{\otimes 2},D)$$
by
 \begin{equation} \label{cent eq1}
\alpha(g,x) = s(x) - {gs}({g^{-1}}{}x)
 \end{equation}
 and
 \begin{equation} \label{cent eq2}
  \beta([2~1] ; x_1,x_2) =  \beta^l(x_1,x_2) = s(x_1 \dashv x_2) - s(x_1) \dashv s(x_2),
 \end{equation}
 \begin{equation}
  \beta([1~2] ; x_1,x_2) =  \beta^r(x_1,x_2) = s(x_1 \vdash x_2) - s(x_1) \vdash s(x_2).
 \end{equation}
We claim that $(\alpha, \beta)\in \tilde{Z}^1_G(D,D)$. Note that $\beta$ is a $2$-cocycle of the dialgebra cohomology defined in \cite{anita gm}. Next, we have
\begin{equation}
\begin{split}
 {g\alpha(h,{g^{-1}}x)} + \alpha(g,x) &= {g(s({g^{-1}}x)} - {ghs({h^{-1}g^{-1}}{}x))} + s(x) - {gs}({g^{-1}}x)\\& = {gs({g^{-1}}x)} - {{gh}s}({h^{-1}g^{-1}}{}x)+ s(x) - {gs}({g^{-1}}x)\\& = s(x)- {{gh}s}({h^{-1}g^{-1}}{}x)
\\&=\alpha(gh,x).
 \end{split}
 \end{equation}
 
To obtain the remaining equations, we have to consider two cases. For $\varepsilon(g) = +1$ from Equation (\ref{cent eq1}), we have
\begin{align*}
 s(x_1\dashv x_2) &= \,gs({g^{-1}}x_1\dashv \,{g^{-1}}x_2) + \alpha(g,x_1\dashv x_2) \\
 &= \, g(s({g^{-1}}x_1)\dashv s({g^{-1}}x_2) + \beta^l({g^{-1}}x_1,{g^{-1}}x_2)) + \alpha(g,x_1\dashv x_2)\\
 &= \,gs({g^{-1}}x_1)\,\dashv gs({g^{-1}}x_2) + \,g\beta^l({g^{-1}}x_1,{g^{-1}}x_2) + \alpha(g,x_1\dashv x_2),
  \end{align*}
 and from Equation (\ref{cent eq2}) we also have
\begin{align*}
 s({x_1}\dashv{x_2}) &= s(x_1)\dashv s(x_2) + \beta^l(x_1,x_2) \\
 &=  ({gs}({g^{-1}}x_1) + i(\alpha(g,x_1))) \dashv ({gs}({g^{-1}}x_2) + i(\alpha(g,x_2))) + \beta^l(x_1,x_2)   \\
  &= {gs}({g^{-1}}x_1)\dashv {gs}({g^{-1}}x_2) + {x_1}\dashv \alpha(g,x_2) + \alpha(g,x_1)\dashv {x_2} + \beta^l(x_1,x_2).
  \end{align*}
  
Comparing these expression we see that
 \begin{equation}
 x_1\dashv \alpha(g,x_2)-\alpha(g,x_1\dashv x_2)+\alpha(g,x_1)\dashv x_2= g\beta^l(\,{g^{-1}}x_1,\,{g^{-1}}x_2) - \beta^l(x_1,x_2).
\end{equation}

Similarly, if $\varepsilon(g) = -1$ and from Equation (\ref{cent eq1}) we have
 \begin{equation*}
\begin{split}
 s(x_1\dashv x_2) &= {{g}s}({g^{-1}}{} x_2\vdash {g^{-1}}{}x_1) +\alpha(g,x_1\dashv x_2) \\&= 
 {g(s({g^{-1}}{}x_2)\vdash s({g^{-1}}{}x_1) + \beta^r({g^{-1}}{}x_2,{g^{-1}}{}x_1))} + \alpha(g,x_1\dashv x_2)\\&= {gs({g^{-1}}{}x_1)}\dashv {gs({g^{-1}}{}x_2)} + {g\beta^r({g^{-1}}{}x_2,{g^{-1}}{}x_1)} + \alpha(g,x_1\dashv x_2).
 \end{split}
 \end{equation*}
  From Equation (\ref{cent eq2}) we also have
 \begin{equation*}
 \begin{split}
 s({x_1}\dashv {x_2}) &= s(x_1)\dashv s(x_2) + \beta^l(x_1,x_2) \\&=  ({gs}({g^{-1}}{}x_1) + \alpha(g,x_1)) \dashv ({gs}({g^{-1}}{}x_2) + \alpha(g,x_2)) + \beta^l(x_1,x_2)  
  \\&= {gs}({g^{-1}}{}x_1)\dashv {gs}({g^{-1}}{}x_2) + {x_1}\dashv \alpha(g,x_2) + \alpha(g,x_1)\dashv {x_2} + \beta^l(x_1,x_2).
 \end{split}
 \end{equation*}
Comparing these expression we see that
 \begin{equation}
{x_1}\dashv \alpha(g,x_2)- \alpha(g,x_1\dashv x_2) + \alpha(g,x_1)\dashv{x_2} = {g\beta^r({g^{-1}}{}x_2,{g^{-1}}{}x_1)} - \beta^l(x_1,x_2) .
 \end{equation}

In a similar way for $y = [1~2]$ and $\varepsilon(g) = + 1$, we have 
 \begin{equation}
{x_1}\vdash \alpha(g,x_2)- \alpha(g,x_1\vdash x_2) + \alpha(g,x_1)\vdash{x_2} = {g\beta^r({g^{-1}}{}x_1,{g^{-1}}{}x_2)}-\beta^r(x_1,x_2),
 \end{equation}
 and for $\varepsilon(g) = - 1$, we have
 \begin{equation}
{x_1}\vdash \alpha(g,x_2)- \alpha(g,x_1\vdash x_2) + \alpha(g,x_1)\vdash{x_2} = {g\beta^l({g^{-1}}{}x_2,{g^{-1}}{}x_1)} - \beta^r(x_1,x_2) ,
 \end{equation}
Hence, we showed that in fact $(\alpha,\beta)\in \tilde{Z}^1_G(D,D)$. One can show that two equivalent extensions give rise to cohomologous $1$-cocycles.

Conversely, starting with $(\alpha,\beta)\in \tilde{Z}^1_G(D,D)$ we define $B=D\oplus D$ with the multiplications and action are given by
\begin{equation*}
(x_1,y_1) \dashv (x_2,y_2) = (x_1 \dashv y_2 + y_1 \dashv x_2 + \beta^l(y_1,y_2), y_1 \dashv y_2),
\end{equation*}
\begin{equation*}
(x_1,y_1) \vdash (x_2,y_2) = (x_1 \vdash y_2 + y_1 \vdash x_2 + \beta^r(y_1,y_2), y_1 \vdash y_2),
\end{equation*}
and 
\begin{equation*}
 {g(x,y)} = ({gx} + \alpha(g,{gy}), {gy}).
\end{equation*}

We claim that $B$ satisfies all properties of oriented dialgebra and defines an extension. Since $\beta$ is a  $2$-cocycle of the dialgebra cohomology, $B$ is clearly a dialgebra. So we only to check the Equations (\ref{orientaion eq1}) and (\ref{orientation eq2}). There are two cases to consider for $\varepsilon(g) = +1$ and $\varepsilon(g) = -1$. Firstly, we deal with the case when $y= [2~1]$ and $\varepsilon(g) = +1$. We have
\begin{equation*}
\begin{split}
{g((x_1,y_1)\dashv(x_2,y_2))} &= {g(x_1\dashv y_2 + y_1\dashv x_2 + \beta^l(y_1,y_2), y_1\dashv y_2)}\\& = ({gx_1}\dashv {gy_2} + {gy_1}\dashv {gx_2} + {g\beta^l(y_1,y_2)} + \alpha(g,{gy_1}\dashv {gy_2}), {{gy_1}\dashv{gy_2}}),
\end{split}
 \end{equation*}
and
\begin{equation*}
  \begin{split}
{g(x_1,y_1)}\dashv {g(x_2,y_2)} &= ({gx_1} + \alpha(g,{gy_1}), {gy_1})\dashv ({gx_2} + \alpha(g,{gy_2}), {gy_2})\\&= ({gx_1}\dashv {gy_2} + \alpha(g,{gy_1})\dashv {gy_2} + {gy_1}\dashv {gx_2} \\&+ {gy_1}\dashv \alpha(g,{gy_2}) + \beta^l({gy_1},{gy_2}),{gy_1}\dashv {gy_2} ).
\end{split}
 \end{equation*}
Therefore, using Equation (\ref{2-cocycle exp 3}) in the last equation, we get
 \[{g((x_1,y_1)\dashv(x_2,y_2))} = {g(x_1,y_1)}\dashv {g(x_2,y_2)}.\]
Next, we deal with the second case when $y= [2~1]$ and $\varepsilon(g) = -1$. We have 
\begin{equation*}
\begin{split}
{g((x_1,y_1)\dashv (x_2,y_2))} &= {g(x_1\dashv y_2 + y_1\dashv x_2 + \beta^l(y_1,y_2), y_1\dashv y_2)}\\& = ({gy_2}\vdash {gx_1} + {gx_2}\vdash {gy_1} + {g\beta^l(y_1,y_2)} + \alpha(g,{gy_2}\vdash {gy_1}), {gy_2}\vdash {gy_1}),
\end{split}
 \end{equation*}
and
\begin{equation*}
\begin{split}
{g(x_2,y_2)}\vdash {g(x_1,y_1)} &= ({gx_2} + \alpha(g,{gy_2}), {gy_2})\vdash ({gx_1} + \alpha(g,{gy_1}), {gy_1})\\&= ({gx_2}\vdash {gy_1} + \alpha(g,{gy_2})\vdash {gy_1} + {gy_2}\vdash {gx_1} \\&+ {gy_2}\vdash \alpha(g,{gy_1}) + \beta^r({gy_2},{gy_1}),{gy_2}\vdash {gy_1} ).
\end{split}
 \end{equation*}
 Therefore, using Equation (\ref{2-cocycle exp 3}) in the last equation it follows that
 \[{g((x_1,y_1)\dashv (x_2,y_2))} = {g(x_2,y_2)}\vdash{g(x_1,y_1)}.\]
 Similar computations for $y = [1~2]$ gives
  \[{g((x_1,y_1)\vdash(x_2,y_2))} = {g(x_1,y_1)}\vdash {g(x_2,y_2)},\]
  \[{g((x_1,y_1)\vdash (x_2,y_2))} = {g(x_2,y_2)}\dashv {g(x_1,y_1)}.\]
  Hence, $B=D\oplus D$ is an oriented dialgebra. Now we define maps $i : D \to  D\oplus D$ and $p: D\oplus D \to D$ as follows:
  $$i(x) = (x, 0),~~~~p(b_1, b_2) = b_2.$$
 It is easy to check that $i$ and $p$ satisfies all the conditions of the definition of singular extensions. 
Thus, one obtains an inverse map from the cohomology to extensions.
 \end{proof}

\section{One-parameter formal deformations of oriented dialgebras}\label{sec 7}
In this final section, we introduce a one-parameter formal deformations of oriented dialgebras. We discuss how deformation is related to the cohomology of oriented dialgebras.
\begin{defn}
Let $(D, \dashv, \vdash)$ be an oriented dialgebra over an oriented group $(G, \varepsilon)$. A one-parameter formal deformation of $D$ is a triple $(\mathfrak{m}^l_t,  \mathfrak{m}^r_t, \Phi_t)$, where 
$$\mathfrak{m}^l_t = \sum^{\infty}_{i=0} m^l_i t^i,~~~\mathfrak{m}^r_t = \sum^{\infty}_{i=0} m^r_i t^i~~~\text{and}~ \Phi_t = \sum^{\infty}_{i=0} \phi_i t^i$$
are formal power series such that
\begin{enumerate}
\item $m^l_i, m^r_i \in \text{Hom}(D \otimes D, D)$ and $\phi_i \in \text{Maps}(G, \text{Hom}(D, D))~~~\text{for all}~ i\geq 0.$
\item $m^l_0(y_1, y_2)=y_1\dashv y_2,~m^r_0(y_1, y_2)=y_1\vdash y_2,~\text{and}~\phi_0(g, x)=gx~\text{for all}~x, y_1, y_2\in D.$ 
\item $(D[[t]],\mathfrak{m}^l_t,\mathfrak{m}^r_t)$ is a dialgebra.
\item $\Phi_t (g_1 g_2, x) = \Phi_t (g_1, \Phi_t(g_2, x))$ for all $g_1, g_2 \in G$ and $x\in D$.
\item $\Phi_t (g, \mathfrak{m}^l_t(y_1, y_2))=\begin{cases}
                                              & \mathfrak{m}^l_t(\Phi_t(g, y_1), \Phi_t(g, y_2))~~~\text{if}~\varepsilon (g)=+1,\\
                                              &  \mathfrak{m}^r_t(\Phi_t(g, y_2), \Phi_t(g, y_1))~~~\text{if}~\varepsilon (g)=-1.
                                             \end{cases}$
\item $\Phi_t (g, \mathfrak{m}^r_t(y_1, y_2))=\begin{cases}
                                              & \mathfrak{m}^r_t(\Phi_t(g, y_1), \Phi_t(g, y_2))~~~\text{if}~\varepsilon (g)=+1,\\
                                              &  \mathfrak{m}^l_t(\Phi_t(g, y_2), \Phi_t(g, y_1))~~~\text{if}~\varepsilon (g)=-1.
                                             \end{cases}$                                             

\end{enumerate}
\end{defn}
For all $n\geq 0$, expanding and equating the coefficients of $t^n$ from both sides of equation (4) in the above definition, we have
$$\phi_n (g_1 g_2, x)= \sum_{i+j=n} \phi_i (g_1, \phi_j (g_2, x)).$$
Similarly for all $n\geq 0$, expanding and equating the coefficients of $t^n$ from both sides of equation (5) in the above definition, we have
\begin{align*}
\sum_{i+j=n} \phi_i (g, m^l_j(y_1, y_2))=\begin{cases}
                                              & \sum_{i+j+k=n}m^l_i(\phi_j (g, y_1), \phi_k (g, y_2))~~~\text{if}~\varepsilon (g)=+1,\\
                                              & \sum_{i+j+k=n} m^r_i (\phi_j (g, y_2), \phi_k (g, y_1))~~~\text{if}~\varepsilon (g)=-1.
                                             \end{cases}
\end{align*}

\begin{align*}
\sum_{i+j=n} \phi_i (g, m^r_j(y_1, y_2))=\begin{cases}
                                              & \sum_{i+j+k=n}m^r_i(\phi_j (g, y_1), \phi_k (g, y_2))~~~\text{if}~\varepsilon (g)=+1,\\
                                              & \sum_{i+j+k=n} m^l_i (\phi_j (g, y_2), \phi_k (g, y_1))~~~\text{if}~\varepsilon (g)=-1.
                                             \end{cases}
\end{align*}
In particular, for $n=1$, we have
\begin{align} \label{df eq1}
\phi_1(g_1 g_2, x) &= \phi_0 (g_1, \phi_1 (g_2, x)) +\phi_1 (g_1, \phi_0 (g_2, x))\\
                              &= g_1 \phi_1 (g_2, x) + \phi_1 (g_1, g_2 x), \nonumber
\end{align}
and 
\begin{align} \label{df eq2}
\phi_0 (g, m^l_1 (y_1, y_2)) + \phi_1 (g, m^l_0 (y_1, y_2))&=\begin{cases}
                                              & m^l_1(\phi_0 (g, y_1), \phi_0 (g, y_2)) + m^l_0(\phi_1(g, y_1), \phi_0 (g, y_2))\\
                                              & + m^l_0(\phi_ 0 (g, y_1), \phi_1 (g, y_2))~~~\text{if}~\varepsilon (g)=+1,\\
                                              & m^r_1 (\phi_0 (g, y_2), \phi_0 (g, y_1))+ m^r_0 (\phi_1 (g, y_2), \phi_0 (g, y_1))\\
                                              & +m^r_0 (\phi_0 (g, y_2), \phi_1 (g, y_1))~~~\text{if}~\varepsilon (g)=-1.   
                                           \end{cases}
\end{align}     
\begin{align}\label{df eq3}
\phi_0 (g, m^r_1 (y_1,y_2)) + \phi_1 (g, m^r_0 (y_1, y_2))&=\begin{cases}
                                              & m^r_1(\phi_0 (g, y_1), \phi_0 (g, y_2)) + m^r_0(\phi_1(g, y_1), \phi_0 (g, y_2))\\
                                              & + m^r_0(\phi_ 0 (g, y_1), \phi_1 (g, y_2))~~~\text{if}~\varepsilon (g)=+1,\\
                                              & m^l_1 (\phi_0 (g, y_2), \phi_0 (g, y_1))+ m^l_0 (\phi_1 (g, y_2), \phi_0 (g, y_1))\\
                                              & +m^l_0 (\phi_0 (g, y_2), \phi_1 (g, y_1))~~~\text{if}~\varepsilon (g)=-1.
                                             \end{cases}
 \end{align}        

We rewrite the Equations in (\ref{df eq2}) and (\ref{df eq3}) as follows:

\begin{align}  \label{df eq4}                                                                               
                                             gm_1^l(y_1, y_2) + \phi_1(g, y_1 \dashv y_2)           &=\begin{cases}
                                              & m^l_1(g y_1, g y_2) + (\phi_1(g, y_1) \dashv g y_2)\\
                                              & + (g y_1 \dashv \phi_1 (g, y_2))~~~\text{if}~\varepsilon (g)=+1,\\
                                              & m^r_1 (g y_2, g y_1)+ (\phi_1 (g, y_2)\vdash g y_1)\\
                                              & + (g y_2\vdash \phi_1 (g, y_1))~~~\text{if}~\varepsilon (g)=-1.
                                             \end{cases}     
\end{align}

 \begin{align}      \label{df eq5}                              
   gm_1^r(y_1, y_2) + \phi_1(g, y_1 \vdash y_2)     &=\begin{cases}
                                              & m^r_1(g y_1, g y_2) + (\phi_1(g, y_1) \vdash g y_2)\\
                                              & + (g y_1 \dashv \phi_1 (g, y_2))~~~\text{if}~\varepsilon (g)=+1,\\
                                              & m^l_1 (g y_2, g y_1)+ (\phi_1 (g, y_2)\dashv g y_1)\\
                                              & + (g y_2\dashv \phi_1 (g, y_1))~~~\text{if}~\varepsilon (g)=-1.
                                             \end{cases}     
\end{align}
\begin{defn}
Let $m_n: \mathbb{K}[Y_2]\otimes (D)^{\otimes 2}\to D$, be defined as,
\begin{align*}
m_n(y;y_1,y_2)=\begin{cases}
                        m_n^{l}(y_1,y_2),\,\,\, \text{if}\,\,\, y=[2\,\,\,1],\\
                        m_n^{r}(y_1,y_2),\,\,\, \text{if}\,\,\, y=[1\,\,\,2],
                          \end{cases}   
\end{align*}
and let $\theta_n : G\to \text{Hom}(\mathbb{K}[Y_1]\otimes D, D) \simeq \text{Hom}(D, D)$ be defined as
$$\theta_n(g, x)= \phi_n(g, g^{-1}x).$$
\end{defn}
For $n=1$, we have a notion of an infinitesimal of the deformation. The pair $(m_1, \theta_1)$ is called the infinitesimal of the deformation of oriented dialgebra.
\begin{thm}\label{infinitesimal}
An infinitesimal of an oriented dialgebra deformation is a $1$-cocycle in $\tilde{H}{^1_G}(D,D)$.
\end{thm}
\begin{proof}
We rewrite the Equations (\ref{df eq1}), (\ref{df eq4}), (\ref{df eq5}) using the Definition of $\theta_1$ and  get the following equations:
\begin{align}\label{df new 1}
\theta_1(g_1 g_2, g_1 g_2 x)= g_1 \theta_1 (g_2, g_2 x) + \theta_1(g_1, g_1 g_2 x).
\end{align}
\begin{align}\label{df new 2}
gm_1^l(y_1, y_2) 
=\begin{cases}
& m^l_1(g y_1, g y_2) + (\theta_1(g, g y_1) \dashv g y_2) + (g y_1 \dashv \theta_1 (g, g y_2)) - \theta_1(g, g y_1 \dashv g y_2)~~~\text{if}~\varepsilon (g)=+1, \\
& m^r_1 (g y_2, g y_1) + (\theta_1 (g, g y_2)\vdash g y_1) + (g y_2\vdash \theta_1 (g, g y_1))- \theta_1(g, g y_2 \vdash g y_1~~~\text{if}~\varepsilon (g)=-1,
\end{cases}     
\end{align}
and
\begin{align}\label{df new 3}
gm_1^r(y_1, y_2)  &=\begin{cases}
                                              & m^r_1(g y_1, g y_2) + (\theta_1(g, g y_1) \vdash g y_2) + (g y_1 \dashv \theta_1 (g, g y_2))- \theta_1(g, g y_1 \vdash g y_2)~~~\text{if}~\varepsilon (g)=+1,\\
                                              & m^l_1 (g y_2, g y_1) + (\theta_1 (g, g y_2)\dashv g y_1) + (g y_2\dashv \theta_1 (g, g y_1)) - \theta_1(g, g y_2 \dashv g y_1)~~~\text{if}~\varepsilon (g)=-1.
                                             \end{cases}     
\end{align}
Comparing Equations (\ref{coho 1}), (\ref{2-cocycle exp 3}), (\ref{2-cocycle exp 4}) to the Equations (\ref{df new 1}), (\ref{df new 2}), (\ref{df new 3}), we can say that $(m_1, \theta_1)$ is  a $1$-cocycle in $\tilde{H}{^1_G}(D,D)$.
\end{proof}
\begin{defn}
Suppose $m^l_i=0,~m^r_i=0$ and $\phi_i=0$ for all $0<i<n$. Then $(m_n, \theta_n)$ is called the $n$th infinitesimal of the deformation.
\end{defn}
Similar to the Theorem \ref{infinitesimal}, we have the following theorem,
\begin{thm}
The $n$th infinitesimal of an oriented dialgebra deformation is a $1$-cocycle in $\tilde{H}{^1_G}(D,D)$.
\end{thm}


Now we discuss the rigidity of the oriented dialgebra deformations.
\begin{defn}
Two oriented dialgebra deformations $(\mathfrak{m}^{l_1}_t, \mathfrak{m}^{r_1}_t, \Phi_t^1)$ and $(\mathfrak{m}^{l_2}_t, \mathfrak{m}^{r_2}_t, \Phi_t^2)$ are said to be equivalent if there $\mathbb{K}[[t]]$-linear automorphism
$$\Psi_t : D[[t]] \to D[[t]]$$
of the form $\Psi_t= \sum^{\infty}_{n=0} \psi_i t^i$ satisfying
\begin{enumerate}
\item $\psi_i \in \text{Hom} (D, D)$ for all $i\geq 0$.
\item $\psi_0= \text{id}$.
\item $\Psi_t (\mathfrak{m}^{l_2}_t(y_1, y_2))= \mathfrak{m}^{l_1}_t(\Psi_t (y_1), \Psi_t (y_2))$ for all $y_1, y_2\in D$.
\item $\Psi_t (\mathfrak{m}^{r_2}_t(y_1, y_2))= \mathfrak{m}^{r_1}_t(\Psi_t (y_1), \Psi_t (y_2))$ for all $y_1, y_2\in D$.
\item $\Psi_t (\Phi_t^2 (g, x))= \Phi_t^1 (g, \Psi_t (x))$ for all $x \in D$ and $g\in G$.
\end{enumerate}
\end{defn}
For all $n\geq 0$, expanding and equating the coefficients of $t^n$ from both sides of equation (3) and (4) in the above definition, we have
\begin{align}\label{equiv 1}
\sum_{i+j=n} \psi_i (m^{l_2}_j (y_1, y_2))= \sum_{i+j+k=n} m^{l_1}_i (\Psi_j (y_1), \Psi_k (y_2)),
\end{align}
\begin{align}\label{equiv 2}
\sum_{i+j=n} \psi_i (m^{r_2}_j (y_1, y_2))= \sum_{i+j+k=n} m^{r_1}_i (\Psi_j (y_1), \Psi_k (y_2)).
\end{align}
Similarly for all $n\geq 0$, expanding and equating the coefficients of $t^n$ from both sides of equation (5) in the above definition, we have
\begin{align}\label{equiv 3}
\sum_{i+j=n} \psi_i (\phi^2_j (g, x))= \sum_{i+j =n}\phi^1_i (g, \psi_j(x))
\end{align}
\begin{defn}
An oriented dialgebra $D$ is called rigid if all the deformations of $D$ are equivalent.
\end{defn}
\begin{prop}
The infinitesimals of two equivalent deformation determines same cohomology class.
\end{prop}
\begin{proof}
Let  $(\mathfrak{m}^{l_1}_t, \mathfrak{m}^{r_1}_t, \Phi_t^1)$ and $(\mathfrak{m}^{l_2}_t, \mathfrak{m}^{r_2}_t, \Phi_t^2)$ are two equivalent deformations.
For $n=1$, Equations (\ref{equiv 1}), (\ref{equiv 2}), (\ref{equiv 3}) gives the following equations:
\begin{align*}
&\psi_0(m_1^{l_2}(y_1, y_2)) + \psi_1(m_0^{l_2}(y_1, y_2))= m_0^{l_1}(\psi_0(y_1), \psi_1(y_2)) + m_0^{l_1}(\psi_1(y_1), \psi_0(y_2)) + m_1^{l_1}(\psi_0(y_1), \psi_0(y_2)),\\
&\psi_0(m_1^{r_2}(y_1, y_2)) + \psi_1(m_0^{r_2}(y_1, y_2))= m_0^{r_1}(\psi_0(y_1), \psi_1(y_2)) + m_0^{r_1}(\psi_1(y_1), \psi_0(y_2)) + m_1^{r_1}(\psi_0(y_1), \psi_0(y_2)),\\
&\psi_0(\phi^2_1(g,x)) + \psi_1(\phi^2_0(g,x))= \phi^1_0(g, \psi_1(x)) + \phi^1_1(g, \psi_0(x)).
\end{align*}
 Rewriting the above equations, we have
\begin{align*}
&m_1^{l_2}(y_1, y_2) - m_1^{l_1}(y_1, y_2) = y_1 \dashv \psi_1(y_2) - \psi_1 (y_1 \dashv y_2) + \psi_1(y_1)\dashv y_2 = \partial^{'} \psi_1([2~1]; y_1, y_2),\\
&m_1^{r_2}(y_1, y_2) - m_1^{r_1}(y_1, y_2) = y_1 \vdash \psi_1(y_2) - \psi_1 (y_1 \vdash y_2) + \psi_1(y_1)\vdash y_2 = \partial^{'} \psi_1([1~2]; y_1, y_2),\\
& \theta^2_1 (g, gx) - \theta^1_1 (g, gx) = (g\psi_1(x) - \psi_1(gx)) = \partial^{''}\psi_1 (g, gx).
\end{align*}
Thus, infinitesimals of two equivalent deformations are cohomologous.
\end{proof}

\begin{remark}
In this work we have studied oriented dialgebras from the cohomological point of view. However, this approach is applicable to various
Loday-type algebras and to their twisted analogues. In our future project, we plan to explicitly describe in a general fashion the cohomology of oriented Loday-type algebras with coefficients in a representation.
\end{remark}

\mbox{ }\\

\providecommand{\bysame}{\leavevmode\hbox to3em{\hrulefill}\thinspace}
\providecommand{\MR}{\relax\ifhmode\unskip\space\fi MR }
\providecommand{\MRhref}[2]{%
  \href{http://www.ams.org/mathscinet-getitem?mr=#1}{#2}
}
\providecommand{\href}[2]{#2}

\mbox{ } \\
\end{document}